\documentclass{amsart}
\usepackage{amssymb,amsmath,amsthm,amsfonts}
  \usepackage{graphics}
  \usepackage{epsfig}
\usepackage{graphicx}  \usepackage{epstopdf}
\usepackage[utf8]{inputenc}
\usepackage[english]{babel}
\usepackage[colorlinks=true]{hyperref}
\hypersetup{urlcolor=blue, citecolor=red}

\theoremstyle{plain}
\newtheorem{teo}{Theorem}[section]
\newtheorem{thm}[teo]{Theorem}

\newtheorem{prop}[teo]{Proposition}		

\theoremstyle{definition}

\newtheorem{rmk}[teo]{Remark}

\DeclareMathOperator{\interior}{int}
\DeclareMathOperator{\dist}{dist}

\newcommand{\expc}{\eta}

\newcommand{\R}    {\mathbb R}

\newcommand{\Z}  {\mathbb Z}

\renewcommand{\epsilon}{\varepsilon}

\newcommand{\fFR}{f_{\text{FR}}}

\author[Alfonso Artigue]{}
\email{artigue@unorte.edu.uy}
\title[Expansive homeomorphisms of 3-manifolds]{An expansive homeomorphism of a 3-manifold with a local stable set that is not locally connected}

\subjclass{Primary: 37B45; Secondary: 37B05.}

\keywords{topological dynamics, expansive homeomorphism, continuum theory}

\begin{document}
\begin{abstract}
In this article we construct an expansive homeomorphism of a compact three-dimensional manifold
with a fixed point whose local stable set is not locally connected. 
This homeomorphism is obtained as a topological perturbation of a quasi-Anosov diffeomorphism that is not Anosov.
\end{abstract}

\maketitle

\centerline{\scshape Alfonso Artigue}
\medskip
{\footnotesize
 \centerline{Departamento de Matem\'atica y Estad\'\i stica del Litoral}
   \centerline{Universidad de la Rep\'ublica}
   \centerline{Gral Rivera 1350, Salto, Uruguay}
}

\section{Introduction}
A homeomorphism $f$ of a metric space $(M,\dist)$ is \emph{expansive} if there is
$\expc>0$ such that if $x,y\in M$ and $\dist(f^n(x),f^n(y))\leq\expc$ for all $n\in\Z$ then
$x=y$. 
Expansivity is a well known property of Anosov diffeomorphisms and hyperbolic sets. 
Also, pseudo-Anosov \cite{Hi,L} and quasi-Anosov diffeomorphisms \cite{FR} are known to be expansive.
In \cite{Hi,L}, Hiraide and Lewowicz proved that on compact surfaces, expansive homeomorphisms are 
conjugate to pseudo-Anosov diffeomorphisms. In particular, if $M$ is the two-torus then $f$ is conjugate to an 
Anosov diffeomorphism and 
there are no expansive homeomorphisms on the two-sphere. 

In \cite{Vi2002} Vieitez proved that if $f$ is an expansive diffeomorphism of a compact three-dimensional manifold and 
$\Omega(f)=M$ then $f$ is conjugate to an Anosov diffeomorphism.
We recall that $x\in M$ is a \emph{wandering point} if there is an open set $U\subset M$ such that $x\in U$ and $f^n(U)\cap U=\emptyset$ 
for all $n\neq 0$. 
The set of non-wandering points is denoted as $\Omega(f)$.
Given a homeomorphism $f\colon M\to M$ and $\epsilon>0$ small, the $\epsilon$-\emph{stable set} (\emph{local stable set}) of a point $x\in M$
is defined as
\[
  W^s_\epsilon(x)=\{y\in M:\dist(f^n(x),f^n(y))\leq\epsilon \hbox{ for all }n\geq 0\}.
\]
On three-dimensional manifolds there are several open problems. For instance, it is not known 
whether the three-sphere admits expansive homeomorphisms. 
In that article Vieitez asks:
allowing wandering points, can we have points with local stable 
sets that are not manifolds for $f\colon M\to M$, an expansive diffeomorphism
defined on a three-dimensional manifold $M$?
We remark that in the papers by Hiraide and Lewowicz that we mentioned, 
before proving that stable and unstable sets form 
pseudo-Anosov singular foliations, they show that local stable sets are locally connected.

The purpose of this paper is to give a positive answer to Vieitez' question for homeomorphisms.
We will construct an expansive homeomorphism on a three-dimensional manifold
with a point whose local stable set is not locally connected 
and $\Omega(f)$ consists of an attractor and a repeller.
The construction follows the ideas in \cite{ArAn} 
where it is proved that the compact surface of genus two admits 
a continuum-wise expansive homeomorphism with a fixed point 
whose local stable set is not locally connected.
The key point for the present construction is that on a three-dimensional manifold there is enough room 
to place two one-dimensional continua meeting in a singleton, even if one of these continua is not locally connected.
In Remark \ref{rmkNoC1} we explain why our example is not $C^1$.

\section{The example}
\label{secAno}

The example is a $C^0$ perturbation of the quasi-Anosov diffeomorphism of \cite{FR}. 
The perturbation will be obtained by a composition with a homeomorphism that is close to the identity. 
This homeomorphism will be defined in local charts around a fixed point and will be extended as the identity outside this chart. 
In \S \ref{secPertR3} we construct this perturbation in $\R^3$ obtaining a homeomorphism such that the stable set of the origin $(0,0,0)$ is 
not locally connected. 
Then, in \S \ref{secPertQA} we perform the perturbation of the quasi-Anosov diffeomorphism to obtain our example.

\subsection{In local charts}
\label{secPertR3}
Fix $r\in (0,1/2)$ and let $B\subset \R^3$ be the closed ball of radius $r$ centered at $(\frac 32,0,0)$.
Let $T_2\colon\R^3\to\R^3$ be defined as 
$$T_2(x,y,z)=\frac 12(x,y,z).$$ 
Define $B_n=T_2^n(B)$ for $n\in\Z$. 
Let $E\subset B$ be the non-locally connected continuum given by the union of the following segments:
\begin{itemize}
 \item a segment parallel to $(1,0,0)$: $\left[\frac 32-\frac r2,\frac32+\frac r2\right]\times\left\{(0,0)\right\}$,
 \item a segment parallel to $(0,1,0)$: $\left\{\frac 32-\frac r2\right\}\times\left[0,\frac r2\right]\times \{0\}$,
 \item a countable family of segments parallel to $(0,1,0)$: 
 $\left\{\frac 32-\frac r2+\frac rk\right\} \times \left[0,\frac r2\right] \times \{0\}$, 
 for $k\geq 1$.
\end{itemize}
%
%
%
Consider $\rho\colon B\to [0,1]$ a smooth function such that $\rho^{-1}(1)=E$ and $\rho^{-1}(0)=\partial B$.
Define a vector field $X\colon\R^3\to\R^3$ by 
\[
 X(x,y,z)=\left\{
 \begin{array}{ll}
  (0,-\rho(T_2^{-n}(x,y,z))y\log 4,0) & \text{ if }\exists n\in\Z\text{ s.t. }(x,y,z)\in B_n,\\
  (0,0,0)&\text{ otherwise.}
 \end{array}
 \right.
\]

\begin{rmk}
 The vector field $X$ induces a flow $\phi\colon\R\times\R^3\to\R^3$. 
 The proof is as follows.
 Notice that $X$ is smooth on $\R^3\setminus \{(0,0,0)\}$. 
 Since the regular orbits (i.e., not fixed points) are contained in the compact balls $B_n$ we have that these trajectories are defined for all $t\in\R$. 
 At the origin we have a fixed point, and the continuity of $X$ gives the continuity of the complete flow on $\R^3$.
\end{rmk}

Let $\phi_t$ be the flow on $\R^3$ induced by $X$. For $t=1$ we obtain the so called \emph{time-one map} $\phi_1$.
Let $$E_k=T_2^k(E)$$ for all $k\in\Z$. 

\begin{rmk}
\label{rmkNoC1}
 The homeomorphism $\phi_1$ is not $C^1$. 
 This is because the partial derivative $\frac{\partial \phi_1}{\partial y} (x_*,0,0)=(0,1/4,0)$ if 
 $(x_*,0,0)$ is in one of the segments of $E_k$ that is parallel to $(0,1,0)$. 
 Since $\frac{\partial \phi_1}{\partial y} (0,0,0)=(0,1,0)$ and $x_*>0$ can be taken arbitrarily small, 
 we conclude that $\frac{\partial \phi_1}{\partial y}$ is not continuous at the origin.
\end{rmk}

Define the homeomorphisms $T_1,f\colon \R^3\to \R^3$ as 
\begin{equation}
\label{ecuDefT1}
T_1(x,y,z)=(x/2, 2y,2z)
\end{equation}
and
\[
f=T_1\circ\phi_1.
\]
Consider the set 
$\tilde W=\cup_{k\in\Z}T_2^k(E)\cup\{(x,0,0):x\in\R\}$.
The \emph{(global) stable set} of a point $a\in\R^3$ associated to the homeomorphism $f$ 
is the set 
\[
 W^s_f(a)=\{b\in\R^3:\dist(f^n(a),f^n(b))\to 0\text{ as }n\to +\infty\}.
\]
\begin{prop}
\label{propConjEst}
 It holds that $W^s_f(0,0,0)=\tilde W$.
\end{prop}

\begin{proof}
We start proving the inclusion $\tilde W\subset W^s_f(0,0,0)$. 
We will show that $f(E_k)=E_{k+1}$. 
By the definition of the flow we have that $$\phi_1(E_k)=\{(x,y/4,0):(x,y,0)\in E_k\}.$$ 
Then $T_1(\phi_1(E_k))=\{(x/2,y/2,0):(x,y,0)\in E_k\}$ and this set is $T_2(E_k)=E_{k+1}$. 
Since $E_k\to (0,0,0)$ as $k\to+\infty$ we conclude that $E_k\subset W^s_f(0,0,0)$ for all $k\in\Z$. 
Also, $f(x,0,0)=(x/2,0,0)$ for all $x\in\R$. This proves that
$\{(x,0,0):x\in\R\}\subset W^s_f(0,0,0)$. 

Take $p\notin \tilde W$. 
Note that if $p\notin B_n$ for all $n\in\Z$ then $f^n(p)=T_1^n(p)$ for all $n\geq 0$ 
and $p\notin W^s_f(0,0,0)$.
Assume that $f^n(p)\in B_n$ for all $n\in\Z$ and define 
$$(a_n,b_n,c_n)=f^n(p).$$ 
If $c_0\neq 0$ then $c_n=2^nc_0\to\infty$. 
Thus, we assume that $c_0=0$. 
Suppose that $b_0>0$ (the case $b_0<0$ is analogous). 
Let $y_0=\sup\{s\geq 0:(a_0,s,0)\in\tilde W\}$. 
Define $l=\{(a_0,y,0):y\in\R\}$ and $g\colon l\to l$ by the equation 
\[
 T_2^{-1}\circ f(a_0,y,0)=(a_0,g(y),0).
\]

In this paragraph we will show that $g(y)>y$ for all $y>y_0$. 
Let $\alpha\colon\R\to\R$ be such that $(a_0,\alpha(t),0)=\phi_t(a_0,y,0)$ for all $t\in\R$. 
That is, $\alpha$ satisfies $\alpha(0)=y$ and 
$\dot\alpha(t)=-\rho(a_0,\alpha(t),0)\alpha(t)\log 4$.
Since $(a_0,y,0)\notin \tilde W$ we have that $$\rho(a_0,\alpha(0),0)<1.$$
Then 
$$\int_0^1\frac{\dot\alpha(t)}{\alpha(t)}dt=-\int_0^1\log 4\rho(a_0,\alpha(t),0)dt>-\log 4$$
and $\log(\alpha(1))-\log(\alpha(0))>-\log 4$. 
Consequently, $\alpha(1)>\frac14\alpha(0)$.
Notice that
\[
\begin{array}{ll}
(a_0,g(y),0)&=T_2^{-1}\circ f(a_0,y,0)=T_2^{-1}\circ T_1\circ\phi_1(a_0,y,0)
=T_2^{-1}\circ T_1(a_0,\alpha(1),0)\\
&=T_2^{-1}(a_0/2,2\alpha(1),0)
=(a_0,4\alpha(1),0).
\end{array}
\]
Then $g(y)=4\alpha(1)$ and $g(y)>y$.

Recall that $X$ is the vector field that defines $\phi$. 
Since $X\circ T_2=T_2\circ X$, as can be easily checked, we have that $f\circ T_2= T_2\circ f$. This implies that 
\[
\begin{array}{ll}
 (a_0,g^n(b_0),0)& =(T_2^{-1}\circ f)^n(a_0,b_0,0)
 = T_2^{-n}\circ f^n(a_0,b_0,0)\\
 &= T_2^{-n}(a_n,b_n,0) 
 = 2^{-n}(a_n,b_n,0)
 \end{array}
\]
and $g^n(b_0)=2^{-n}b_n$. 
Since we are assuming that $f^n(p)\in B_n$ for all $n\in\Z$, we have that $2^{-n}b_n$ is bounded. But, as $b_0>y_0$ and $g(y)>y$ for all $y>y_0$ we have that 
$g^n(b_0)$ is increasing and bounded. 
Then, if $b_*>y_0$ is the limit of $g^n(b_0)$ we have that $g(b_*)=b_*$, 
which contradicts that $g(y)>y$ for all $y>y_0$. 
This implies that $f^n(p)$ cannot be in $B_n$ for all $n\geq 0$, and as we said, this shows that $p\notin W^s_f(0,0,0)$. 
This proves the inclusion $W^s_f(0,0,0)\subset \tilde W$.
\end{proof}

\begin{rmk}
By the definition of the vector field $X$ we see that $\phi_1$ 
preserves the horizontal planes (i.e., the planes perpendicular to $(0,0,1)$). 
Also, $\phi_1$ leaves invariant the cube $[-2,2]^3$ and is the identity in its boundary.
\end{rmk}

\subsection{The local perturbation of the quasi-Anosov}
\label{secPertQA}
To construct our example we start with a quasi-Anosov diffeomorphism as in \cite{FR}. 
A \emph{quasi-Anosov diffeomorphism} is an axiom A diffeomorphism of $M$ such that $T_x W^s(x)\cap T_x W^u(x)=0_x$ for all $x\in M$, 
where $T_x W^\sigma(x)$, $\sigma=s,u$, denotes the tangent space of the stable or unstable manifold $W^\sigma(x)$ at $x$ and $0_x$ is the null vector of $T_x M$.
The quasi-Anosov diffeomorphism of \cite{FR}, that will be denoted as $\fFR\colon M\to M$, has the following properties: 
it is defined on a three-dimensional manifold, it is not Anosov and its non-wandering set is the union of two basic sets. 
The basic sets are an expanding attractor and a shrinking repeller. Both sets are two-dimensional and locally they are 
homeomorphic to the product of $\R^2$ and a Cantor set. 

On the attractor there is a hyperbolic fixed point $p$. 
%
%
Take closed balls $U,V\subset M$ and
$C^0$ local charts 
$\varphi\colon [-2,2]^3\to U\subset M$ and $\psi\colon [-1,1]\times [-4,4]^2\to V\subset M$ satisfying the following conditions: 
\begin{enumerate}
 \item[C0:] $\varphi|_{[-1,1]\times[-2,2]^2}=\psi|_{[-1,1]\times[-2,2]^2}$,
 \item[C1:] $\fFR|_U=\psi\circ T_1\circ\varphi^{-1}$ where $T_1$ was defined in \eqref{ecuDefT1},
 \item[C2:] in the local charts stable sets of $\fFR$ are lines parallel to $(1,0,0)$,
 \item[C3:] there is $r\in (0,1/2)$ such that $W^u_{\fFR}(q)$ in the local chart 
 is transverse to the horizontal planes if $\varphi^{-1}(q)$ is in a neighborhood of $B$ where, as in \S \ref{secPertR3}, $B\subset \R^3$ is the 
 ball of radius $r$ centered at $(3/2,0,0)$,
 \item[C4:] if
$
 \tilde B_n=\varphi(B_n)
$
for all $n\geq 0$, we assume that $\fFR^k(\tilde B_0)\cap \tilde B_0=\emptyset$ for all $k\geq 1$.
\end{enumerate}
Let $\tilde\phi\colon M\to M$ be the homeomorphism given by 
\[
 \tilde\phi(x)=\left\{
 \begin{array}{ll}
  \varphi\circ \phi_1\circ\varphi^{-1}(x) & \text{ if } x\in U,\\
  x & \text{ if } x\notin U,
 \end{array}
 \right.
\]
where $\phi_1$ is the time-one of the flow induced by the vector field $X$ of \S \ref{secPertR3}.
Define the homeomorphism $\tilde f\colon M\to M$ as 
$$\tilde f=\fFR\circ\tilde\phi$$ 
Notice that 
\begin{equation}
 \label{ecuIgFrT}
 \tilde \phi(x)=x\text{ for all }x\notin \cup_{n\geq 0} \tilde B_n.
\end{equation}
%
\begin{rmk}
This implies that if $r$ is small then $\tilde f$ is close to $\fFR$ in the $C^0$ topology of homeomorphisms of $M$.
\end{rmk}

\begin{thm}
The homeomorphism $\tilde f\colon M\to M$ is expansive and the local stable set of the fixed point $p$ is connected but not locally connected,		
\end{thm}

\begin{proof}
By Proposition \ref{propConjEst} and the condition C1 we have that the local stable set of $p$ is $\varphi(\tilde W)$.
Since $\tilde W$ is not locally connected and $\varphi$ is a homeomorphism we conclude that the 
local stable set of $p$ is not locally connected.

Let us show that $\tilde f$ is expansive. 
By \eqref{ecuIgFrT} we have that $\tilde f$ and $\fFR$ coincide on $M\setminus \cup_{n\geq 0} \interior(\tilde B_n)$.
Therefore, if $\tilde f^n(x)\notin \cup_{k\geq 0}\tilde B_k$ for all $n\in\Z$ then 
$\tilde f^n(x)=\fFR^n(x)$ for all $n\in\Z$.
Let $\expc_1$ be an expansivity constant of $\fFR$. 
Thus, if $x,y\in M$ are such that $\tilde f^n(x),\tilde f^n(y)\notin \cup_{k\geq 0}\tilde B_k$ for all $n\in\Z$ 
then $\sup_{n\in\Z}\dist(\tilde f^n(x),\tilde f^n(y))>\expc_1$.

From our analisys on local charts of \S \ref{secPertR3} we have that
$\tilde B_{n+1}\subset \fFR(\tilde B_n)=\tilde f(\tilde B_n)$ for all $n\geq 0$.
This and condition C4 implies that
if $\tilde f^{n_0}(x)\in \tilde B_{k_0}$ for some $n_0\in\Z$ and $k_0\geq 0$ 
then $\tilde f^{n_0-k_0}(x)\in \tilde B_0$.
Let $\expc\in (0,\expc_1)$ be such that if $x\in \tilde B_0$ and $\dist(y,x)<\expc$ then, 
in the local chart, $W^s_\expc(x)$ is contained in a horizontal plane and $W^u_\expc(y)$ is transverse to the horizontal planes. 
We have applied conditions C2 and C3.
This implies that $W^s_\expc(x)\cap W^u_\expc(y)$ contains at most one point. 
This proves that $\expc$ is an expansivity constant of $\tilde f$.
\end{proof}

\renewcommand{\refname}{REFERENCES}

\end{document}